\documentclass[a4paper,12pt]{article}
\usepackage{a4wide}
\usepackage{amsmath}
\usepackage{amssymb}
\usepackage{amsthm}
\usepackage{latexsym}
\usepackage{graphicx}
\usepackage[english]{babel}
\usepackage{makeidx}

\newtheorem{obs} [subsection]{Remark}
\newtheorem{exm} [subsection]{Example}

\newtheorem{prop}[subsection]{Proposition}

\newtheorem{teor}[subsection]{Theorem}
\newtheorem{lema}[subsection]{Lemma}
\newtheorem{cor} [subsection]{Corollary}

\def\sdepth{\operatorname{sdepth}}
\def\depth{\operatorname{depth}}
\begin{document}
\selectlanguage{english}
\frenchspacing

\large
\begin{center}
\textbf{Stanley depth of monomial ideals in three variables}

Mircea Cimpoea\c s
\end{center}
\normalsize

\begin{abstract}
We show that $\depth(S/I)=0$ if and only if $\sdepth(S/I)=0$, where $I\subset S=K[x_1,\ldots,x_n]$ is a monomial ideal. We
give an algorithm to compute the Stanley depth of $S/I$, where $I\subset S=K[x_1,x_2,x_3]$ is a monomial ideal. Also, we prove that a monomial ideal $I\subset K[x_1,x_2,x_3]$ minimally generated by three monomials has $\sdepth(I)=2$.

\vspace{5 pt} \noindent \textbf{Keywords:} Stanley depth, monomial ideal.

\vspace{5 pt} \noindent \textbf{2000 Mathematics Subject
Classification:}Primary: 13H10, Secondary: 13P10.
\end{abstract}

\section*{Introduction}

Let $K$ be a field and $S=K[x_1,\ldots,x_n]$ the polynomial ring over $K$.
Let $J\subset I\subset S$ be two monomial ideals. A Stanley decomposition of $I/J$ is a decomposition $\mathcal D: I/J = \bigoplus_{i=1}^ru_i K[Z_i]$ as $K$-vector space, where $u_i\in S$ are monomials and $Z_i\subset\{x_1,\ldots,x_n\}$. We denote $\sdepth(\mathcal D)=min_{i=1}^r |Z_i|$ and $\sdepth(I/J)=max\{\sdepth(\mathcal D)|\;\mathcal D$ is a Stanley decomposition of $I/J\}$. The number $\sdepth(I/J)$ is called the \emph{Stanley depth} of $I/J$. Herzog, Vladoiu and Zheng show in \cite{hvz} that this invariant can be computed in a finite number of steps.
There are two important particular cases. If $I\subset S$ is a monomial ideal, we are interested in computing $\sdepth(S/I)$ and $\sdepth(I)$. There are some papers regarding this problem, like \cite{hvz},\cite{asia},\cite{sum}, \cite{shen} and \cite{mir}. Stanley's conjecture says that $\sdepth(S/I)\geq \depth(S/I)$, or in the general case, $\sdepth(M)\geq \depth(M)$, where $M$ is a finitely generated multigraded $S$-module. The Stanley conjecture for $S/I$ was proved for $n\leq 5$ and in other special cases, but it remains open in the general case. See for instance, \cite{apel}, \cite{hsy}, \cite{jah}, \cite{pops} and \cite{popi}.

Let $I\subset S$ be a monomial ideal. We assume that $G(I)=\{v_1,\ldots,v_m\}$, where $G(I)$ is the set of minimal monomial generators of $I$. We denote $g(I)=|G(I)|$, the number of minimal generators of $I$. Let $v=GCD(u|\; u\in G(I))$. It follows that
$I=vI'$, where $I'=(I:v)$. For a monomial $u\in S$, we denote $supp(u)=\{x_i:\;x_i|u\}$. We denote $supp(I)=\{x_i:\;x_i|u$ for some $u\in G(I)\}$. We denote $c(I)=|supp(I')|$. 
In the first section, we prove results regarding some relations between $\sdepth(S/I)$, $\sdepth(I)$, $g(I)$ and $c(I)$. 
We prove that $depth(S/I)=0$ if and only if $\sdepth(S/I)=0$, see Corollary $1.6$. Thus, the Stanley's conjecture is true for $S/I$, when $\sdepth(S/I)=0$. In the second section, we give an algorithm to compute $\sdepth(S/I)$, where $I\subset S=K[x_1,x_2,x_3]$ is a monomial ideal, see Theorem $2.3$. We prove that a monomial ideal $I\subset K[x_1,x_2,x_3]$ minimally generated by three monomials has $\sdepth(I)=2$, see Theorem $2.4$. Also, if $I\subset K[x_1,x_2,x_3]$ is saturated, we prove that $\sdepth(I)\geq 2$, see Proposition $2.8$.

\textbf{Aknowledgements}. The author would like to express his gratitude to the organizers of PRAGMATIC 2008, Catania, Italy and especially to Professor Jurgen Herzog.

\footnotetext[1]{This paper was supported by CNCSIS, ID-PCE, 51/2007}

\section{Preliminaries results}

Firstly, we recall the following result of Herzog, Vladoiu and Zheng.

\begin{lema}\cite[Lemma 3.6]{hvz}
Let $J\subset I$ be monomial ideals of $S=K[x_1,\ldots,x_n]$, and let $T=S[x_{n+1}]$. Then $\sdepth(IT/JT)=\sdepth(I/J)+1$.
\end{lema}

For any monomial ideal $J\subset S$, we denote $J^c$ the $K$-vector space spanned by all the monomials not contained in $J$.
With this notation, we have the following lemma.

\begin{lema}
Let $I\subset S=K[x_1,\ldots,x_n ]$ be a monomial ideal and $u\in S$ a monomial. Then
\[ I = ((u)^c\cap I)\oplus u(I:u).\]
Also, $I^c = (v)^c \oplus v(I')^c$, where $v=GCD(u|\;u\in G(I))$ and $I'=(I:v)$.
\end{lema}

\begin{proof}
We have $I = I \cap S = I \cap ((u)^c\oplus (u))\cap I = (u)^c\cap I)\oplus ((u)\cap I)$. In order to complete the proof, it is enough to show that $((u)\cap I)=u(I:u)$. Indeed, if $v\in (u)\cap I$ is a monomial, then $v=uw$ for some monomial $w\in S$. Moreover, since
$uw=v\in I$ it follows that $w\in (I:u)$ and thus $v\in u(I:u)$. The inclusion $((u)\cap I) \supseteq u(I:u)$ is similar.

We have $I^c = I^c \cap S = I^c \cap ((v)^c\oplus (v)) = (I^c\cap(v)^c)\oplus (I^c\cap (v))= (v)^c \oplus v(I')^c$.
\end{proof}

\begin{prop}
Let $I\subset S$ be a monomial ideal which is not principal. Then:

(1) $\sdepth(S/I)=\sdepth(S/I')$.

(2)	$\sdepth(I)=\sdepth(I')$.	

(3)	$\sdepth(S/I)\geq n-c(I)$.

(4)	$\sdepth(I)\geq n-c(I)+1$.

(5) $\sdepth(S/I)\geq n-g(I)$.

(6)	$\sdepth(I)\geq \max\{1,n-g(I)+1\}$.
\end{prop}

\begin{proof}
(1) By Lemma $1.2$, $S/I = I^c=(v)^c \oplus vI'^c$, where $v=GCD(u|\;u\in G(I))$ and $I'=(I:v)$. Given a Stanley decomposition $S/I'=\bigoplus_{i=1}^r u'_iK[Z_i]$, then $\bigoplus_{i=1}^r vu'_iK[Z_i]$ is a Stanley decomposition of $v(I'^c)$. On the other hand, one can easily give a Stanley decomposition $\mathcal D$ of $(v)^c$ with $sdepth(\mathcal D)=n-1$. Thus, we obtain a Stanley decomposition of $S/I$ with its Stanley depth $\geq \sdepth(S/I')$. It follows that $\sdepth(S/I)\geq \sdepth(S/I')$.

In order to prove the converse inequality, we give $S/I=\bigoplus_{i=1}^r u_i K[Z_i]$ a Stanley decomposition of $S/I$. It follows that $v(I'^c) = \bigoplus_{i=1}^r (u_i K[Z_i]\cap v(I'^c)) = \bigoplus_{i=1}^r (u_i K[Z_i]\cap (v))$. We claim that $u_i K[Z_i]\cap (v)\neq (0)$ implies $LCM(u_i,v)\in u_iK[Z_i]$. 

Indeed, if $LCM(u_i,v)\notin u_i K[Z_i]$ it follows that $v/GCD(u_i,v)\notin K[Z_i]$ and therefore, there exists $x_j| v/GCD(u_i,v)$ such that $x_j\notin Z_i$. Thus, $v$ cannot divide any monomial of the form $u_iy$, where $y\in K[Z_i]$ and therefore $u_i K[Z_i]\cap (v)=(0)$, a contradiction. We have proved that $LCM(u_i,v)\in u_i K[Z_i]$. This implies $LCM(u_i,v)K[Z_i]\subset u_iK[Z_i]$. Obviously, $LCM(u_i,v)K[Z_i]\subset (v)$ and thus $LCM(u_i,v)K[Z_i]\subset u_i K[Z_i]\cap (v)$.

\pagebreak

On the other hand, if $u\in u_i K[Z_i]\cap (v)$ is a monomial, it follows that $u_i|u$ and $v|u$ and therefore, $LCM(u_i,v)|u$. 
Since $u\in u_i K[Z_i]$, it follows that $u=u_i\cdot w_i$, where $supp(w_i)\subset Z_i$. Moreover, $supp(u/LCM(u_i,v))\subset Z_i$ and thus, $u\in LCM(u_i,v)K[Z_i]$. We obtain $LCM(u_i,v)K[Z_i] = u_i K[Z_i]\cap (v)$.
In conclusion, \[ vI'^c = \bigoplus_{(v)\cap u_iK[Z_i]\neq 0}LCM(u_i,v)K[Z_i]\;so\; I'^c = \bigoplus_{(v)\cap u_iK[Z_i]\neq 0}\frac{u_i}{GCD(u_i,v)}K[Z_i]. \]
It follows that $\sdepth(S/I')\geq \sdepth(S/I)$, as required.

(2) Suppose $I=\bigoplus_{i=1}^r u_iK[Z_i]$ is a Stanley decomposition for $I$. One can easily show that $I'=\bigoplus_{i=1}^r u_i/v K[Z_i]$ is a Stanley decomposition for $I'$, and thus $\sdepth(I)\leq sdepth(I')$. Conversely, if $I'=\bigoplus_{i=1}^r \bar{u}_i K[Z_i]$ is a Stanley decomposition of $I'$ it follows that $I=\bigoplus_{i=1}^r \bar{u}_iv K[Z_i]$ is a Stanley decomposition of $I$.

(3) By $(1)$, we can assume that $I'=I$. By reordering the variables, we can assume that $I\subset (x_1,x_2,\ldots,x_m)$, where $m=c(I)$. We write $I=(I\cap K[x_1,\ldots,x_m])S$. Lemma $1.1$ implies $\sdepth(S/I) = \sdepth(K[x_1,\ldots,x_m]/(I\cap K[x_1,\ldots,x_m])) + n-m\geq n-m$. 

(4) The proof is similar with $(3)$.

(5) In order to prove, we use a strategy similar with the Janet's algorithm, see \cite{imran}. As in the proof of \cite[Proposition 3.4]{hvz}, we use induction on $n\geq 1$. If $n=1$ there is nothing to prove. If $m=1$, $I$ is principal and thus $\sdepth(S/I)=n-1$. Suppose $n>1$ and $m>1$. Let $q=deg_{x_n}(I):=\max\{j:\; x_n^j|u$ for some $u\in G(I)\}$. For all $j\leq q$, we denote $I_j$ the monomial ideal in $S'=K[x_1,\ldots,x_{n-1}]$ such that $I\cap x_n^jS'=x_n^j I_j$. Note that $g(I_j)<g(I)$ for all $j<q$ and $g(I_q)\leq g(I)$. We have
\[ S/I = S'/I_0 \oplus x_n(S'/I_1) \oplus \cdots \oplus x_n^{q-1}(S'/I_{q-1})\oplus x_n^q (S'/I_q)[x_n].\]
It follows that $\sdepth(S/I)\geq \min\{\sdepth(S'/I_j),j<q, \sdepth(S'/I_q)+1\}$. By induction hypothesis, it follows that
$\sdepth(S'/I_j)\geq n-1-q(I_j) \geq n-1-(m-1)=n-m$ for all $j<q$. Also, $\sdepth(S'/I_q)\geq n-1-g(I_q)\geq n-1-m$. This complete the proof.

(6) See \cite[Proposition 3.4]{hvz}.
\end{proof}

\begin{prop}
Let $I\subset S$ be a monomial ideal which is not principal with $c(I)=2$ or $g(I)=2$. Then $\sdepth(I)=n-1$ and $\sdepth(S/I)=n-2$.
\end{prop}

\begin{proof}
If $c(I)=2$, then, by $1.3(4)$, it follows that $\sdepth(I)\geq n-c(I)+1=n-1$. Similarly, $\sdepth(I)\geq n-1$ if $g(I)=2$. But $\sdepth(I)<n$, otherwise, $I$ is principal. 

Also, by $1.3$, $\sdepth(S/I)\geq n-2$ if $c(I)=2$ or $g(I)=2$. We consider the case $c(I)=2$. By $1.3(2)$, we can assume that $I=I'$ and $supp(I)=\{x_1,x_2\}$. Since $GCD(u|\;u\in G(I))=1$, we can assume that $x_1^a\in G(I)$ for some positive integer $a$. Let $w=x_1^{a-1}$. Obviously,
$w\notin I$, but $x_1w\in I$ and $x_2^kw\in I$ for $k\gg 0$. It follows that $w$ is contained in a Staley space of $S/I$ with dimension $\leq n-2$ and thus $\sdepth(S/I)=n-2$.

We consider now the case $g(I)=2$. Suppose $I=(u_1,u_2)$. By $1.3(3)$, we can assume $GCD(u_1,u_2)=1$. Therefore, $I$ is a complete intersection and by \cite[Proposition 1.2]{hsy} or \cite[Corollary 1.4]{asia}, it follows that $\sdepth(S/I)=n-2$.
\end{proof}

\begin{teor}
Let $I\subset S$ be a monomial ideal. Then $\sdepth(S/I)=0$ if and only if $I\neq I^{sat}$, where $I^{sat}=\bigcup_{k\geq 1}(I:(x_1,\ldots,x_n)^k)$ is the \emph{saturation} of $I$.
\end{teor}

\begin{proof}
Suppose $I\neq I^{sat}$ and take $u\in I^{sat}\setminus I$ a monomial. Let $S/I=\bigoplus_{i=1}^r u_iK[Z_i]$ be a Stanley decomposition of $S/I$. Since $u\notin I$ it follows that $u\in u_i K[Z_i]$ for some $i\in [r]$. If $x_j\in Z_i$, then $x_j^k u\in u_iK[Z_i]\subset I^c$ for any $k>0$, a contradiction. Therefore, $u=u_i$ and $Z=\emptyset$ and thus $\sdepth(S/I)=0$.

In order to prove the converse, we use induction on $n\geq 1$. The case $n=1$ is trivial. Suppose $n>1$. We use the decomposition of $S/I$ given by the Janet's algorithm, see \cite{imran}. Let $q=deg_{x_n}(I):=\max\{j:\; x_n^j|u$ for some $u\in G(I)\}$. For all $j\leq q$, we denote $I_j$ the monomial ideal in $S'=K[x_1,\ldots,x_{n-1}]$ such that $I\cap x_n^jS'=x_n^j I_j$. With these notation, we have:
\[(*)\; S/I = S'/I_0 \oplus x_n(S'/I_1) \oplus \cdots \oplus x_n^{q-1}(S'/I_{q-1})\oplus x_n^q (S'/I_q)[x_n].\]
Since $\sdepth(S/I)=0$, it follows that $\sdepth(S'/I_j)=0$ for some $j<q$, otherwise, from the above decomposition it will follow that $\sdepth(S/I)>0$, a contradiction. By induction hypothesis, there exists a monomial $u\in I_j^{sat}\setminus I_j$. We consider the monomial $w=x_n^ju$. 
Since $u\in I_j^{sat}\setminus I_j$ it follows that $w\notin I$ and $x_j^k w \in I$ for $k\gg 0$ and $j<n$. If $x_n^k u\in I$ for some $k\gg 0$ we are done.

Now, suppose that for any $0\leq j\leq q$ and for any monomial $u\in I_j^{sat}\setminus I_j$ and for any positive integer $k$ it follows that $x_n^k u\notin I$. For any $0\leq j\leq q$, we denote $A_j$ the set of monomials which are in $I_j^{sat} \setminus I_j$. By our assumption, we have $A_0\subseteq A_1\subseteq \cdots \subseteq A_q$. Indeed, if $u\in A_j$ for some $j<q$, then $x_n^{j+1}u\notin I$ and so $u\notin I_{j+1}$. On the other hand, $I_j\subset I_{j+1}$ and thus $u\in I_{j+1}^{sat}$. Therefore $u\in A_{j+1}$ and thus $A_j\subset A_{j+1}$.

Since $(I_j^{sat})^{sat}=I_j^{sat}$, by induction hypothesis, it follows that $\sdepth(S'/I_j^{sat})\geq 1$. Suppose $\mathcal D_j: S'/I_j^{sat} = \bigoplus_{i=1}^{r_j}u_{ij}K[Z_{ij}]$ is a Stanley decomposition with $\sdepth(\mathcal D_j)\geq 1$. It follows that $S'/I_j = \bigoplus_{i=1}^{r_j}u_{ij}K[Z_{ij}] \oplus \bigoplus_{u\in A_j}uK$ is a Stanley decomposition of $S'/I_j$. By $(*)$, it follows that
\[ S/I = \bigoplus_{j=0}^{q-1}(\bigoplus_{i=1}^{r_j}x_n^j u_{ij}K[Z_{ij}]\oplus \bigoplus_{u\in A_j}x_n^j uK) \oplus \bigoplus_{i=1}^{r_q}(x_n^q u_{iq}K[Z_{iq},x_n]\oplus \bigoplus_{u\in A_q}x_n^q uK[x_n]).\]
Note that $A_0\subseteq A_1\subseteq \cdots \subseteq A_q$ implies \[ \bigoplus_{j=0}^{q-1}( \bigoplus_{u\in A_j}x_n^j uK) \oplus  \bigoplus_{u\in A_q}x_n^q uK[x_n] = \bigoplus_{j=0}^q \bigoplus_{u\in A_j\setminus A_{j-1}}x_n^j uK[x_n],\] where $A_{-1}=\emptyset$, and thus
\[ S/I = \bigoplus_{j=0}^{q-1}\bigoplus_{i=1}^{r_j}x_n^j u_{ij}K[Z_{ij}]\oplus \bigoplus_{i=1}^{r_q} x_n^q u_{iq}K[Z_{iq},x_n] \oplus \bigoplus_{j=0}^q \bigoplus_{u\in A_j\setminus A_{j-1}}x_n^j uK[x_n], \]
is a Stanley decomposition of $S/I$ and therefore, $\sdepth(S/I)>0$, a contradiction.
\end{proof}

\begin{cor}
Let $I\subset S$ be a monomial ideal. Then $\sdepth(S/I)=0$ if and only if $\depth(S/I)=0$. In particular, $\sdepth(S/I)=0$ if and only if
$\sdepth(S/I^k)=0$, where $k\geq 1$.
\end{cor}

\begin{proof}
Since $I$ is a monomial ideal, it follows that $\depth(S/I)=0$ if and only if $(I:(x_1,\ldots,x_n))\neq I$ which is equivalent, by the previous theorem, with $\sdepth(S/I)=0$. For the second assertion, note that $\depth(S/I)=0$ if and only if $\depth(S/I^k)=0$, where $k\geq 1$.
\end{proof}

\begin{prop}
Let $I\subset S$ be a monomial ideal with $c(I)=n$ and $(x_1,\ldots,x_{n-1})\subset \sqrt{I}$. Then $\sdepth(S/I)=0$.
\end{prop}

\begin{proof}
Since $(x_1,\ldots,x_{n-1})\subset \sqrt{I}$ it follows that $x_j^{a_j}\in I$ for some positive integers $a_j$, where $j\in [n-1]$.
Since $c(I)=n$ it follows that there exists a monomial $u\in G(I)$ with $x_n|u$. If $u=x_n^{a_n}$ there is nothing to prove, since, in this case, $I$ is Artinian. Suppose this is not the case. We consider $w=u/x_n$. Obviously, $x_j^{a_j}w\in I$ for any $j\in [n]$, where $a_n:=1$. Thus, for any Stanley decomposition of $S/I$, the monomial $w$ lays in a Stanley space of dimension $0$ and therefore $\sdepth(S/I)=0$.
\end{proof}

\section{Stanley depth for monomial ideals in three variables}

\begin{prop}
Let $I\subset K[x_1,x_2,x_3]$ be a monomial ideal with $g(I)=3$ and $c(I)=3$. 

(1) If $\sqrt{I}\supseteq (x_1,x_2)$ then $\sdepth(S/I)=0$.

(2) If $\sqrt{I}=(x_1,x_2x_3)$ and $I=(x_1^{a_1}, x_2^{b_1}x_3^{c_1}, x_1^{a_2}x_2^{b_2}x_3^{c_2})$ then $\sdepth(S/I)=0$ if and only if
$b_1>b_2$ and $c_1<c_2$ or $b_1<b_2$ and $c_1>c_2$. Otherwise, $\sdepth(S/I)=1$.

(3) If $\sqrt{I}=(x_1x_2,x_1x_3,x_2x_3)$ and $I=(x_1^{a_1}x_2^{b_1}, x_1^{a_2}x_3^{c_1}, x_2^{b_2}x_3^{c_2} )$ for some positive integers $a_1,a_2,b_1,b_2,c_1,c_2$, then $\sdepth(S/I)=1$ if and only if $a_2\leq a_1-1, b_1\leq b_2-1$ and $c_2\leq c_1-1$ or $a_1\leq a_2-1, b_2\leq b_1-1$ and $c_1\leq c_2-1$.
\end{prop}

\begin{proof}
(1) Is a particular case of Proposition $1.7$.

(2) Suppose $b_1<b_2$ and $c_1>c_2$. We consider $w=x_1^{a_2}x_2^{b_1}x_3^{c_2}$. Obviously, $w\notin I$, but
$x_1^{a_1-a_2}w, x_2^{b_2-b_1}w, x_3^{c_1-c_2}w\in I$. Therefore, the monomial $w$ lays in a Stanley space of $S/I$ of dimension $0$ and thus $\sdepth(S/I)=0$. The case $b_1>b_2$ and $c_1<c_2$ is similar.

Since $x_1^{a_1}, x_2^{b_1}x_3^{c_1}, x_1^{a_2}x_2^{b_2}x_3^{c_2}$ are the minimal generators of $I$ it follows that
$a_1>a_2$ and $b_1>b_2$ or $c_1>c_2$. It is enough to consider the case $b_1>b_2$ and $c_1\geq c_2$. By a straightforward computation, $(I:(x_1,x_2,x_3))=I$, and therefore, by Theorem $1.5$, $\sdepth(S/I)=1$.

(3) We have $(I:x_1)=(x_1^{a_1-1}x_2^{b_1}, x_1^{a_2-1}x_3^{c_1}, x_2^{b_2}x_3^{c_2})$, $(I:x_2)=(x_1^{a_1}x_2^{b_1-1}, x_1^{a_2}x_3^{c_1}, x_2^{b_2-1}x_3^{c_2})$ and $(I:x_3)=(x_1^{a_1}x_2^{b_1}, x_1^{a_2}x_3^{c_1-1}, x_2^{b_2}x_3^{c_2-1})$. Since $(I:(x_1,x_2,x_3))=\bigcap_{i=1}^3(I:x_i)$, we get \[ (I:(x_1,x_2,x_3)) = I + (x_1^{a_1-1}x_2^{b_1}, x_1^{a_2-1}x_3^{c_1})\cap (x_1^{a_1}x_2^{b_1-1},x_2^{b_2-1}x_3^{c_2})\cap (x_1^{a_2}x_3^{c_1-1}, x_2^{b_2}x_3^{c_2-1}) = \]\[ =  I + ( x_1^{a_1-1}x_2^{max\{b_1,b_2-1\}}x_3^{c_2} , x_1^{max\{a_1,a_2-1\}}x_2^{b_1-1}x_3^{c_1}, x_1^{a_2-1}x_2^{b_2-1}x_3^{max\{c_1,c_2\}}) \cap (x_1^{a_2}x_3^{c_1-1}, x_2^{b_2}x_3^{c_2-1}) \] $= I+ (x_1^{max\{a_1-1,a_2\}}x_2^{\max\{b_1,b_2-1\}}x_3^{\max\{c_1-1,c_2\}},  
x_1^{max\{a_1,a_2-1\}}x_2^{max\{b_1-1,b_2\}}x_3^{max\{c_1,c_2-1\}} )$. 

Note that $x_1^{max\{a_1-1,a_2\}}x_2^{\max\{b_1,b_2-1\}}x_3^{\max\{c_1-1,c_2\}}\in I$ if and only if $a_2\leq a_1-1, b_1\leq b_2-1$ and $c_2\leq c_1-1$. Also, 
$x_1^{max\{a_1,a_2-1\}}x_2^{max\{b_1-1,b_2\}}x_3^{max\{c_1,c_2-1\}}\in I$ if and only if $a_1\leq a_2-1, b_2\leq b_1-1$ and $c_1\leq c_2-1$.
Thus, by Theorem $1.5$, we are done.
\end{proof}

\begin{exm}
Let $I=(x_1^2x_2,x_2^2x_3,x_3^2x_1)$ and $J=(I,x_1x_2x_3)$. One can easyly see that $J=(I:(x_1,x_2,x_3))=(J:(x_1,x_2,x_3))$. Thus, by $1.5$, $\sdepth(S/I)=0$ and $\sdepth(S/J)=1$.
\end{exm}

This proposition and the results of the first section, give an algorithm to compute the $\sdepth(S/I)$, where $I\subset S=K[x_1,x_2,x_3]$ is a monomial ideal.

\begin{teor}
Let $I\subset S=K[x_1,x_2,x_3]$ be a monomial ideal.

(a) If $g(I)=1$ then $\sdepth(S/I)=2$. If $g(I)=2$ then $\sdepth(S/I)=1$.

(b) If $g(I)\geq 3$ and $c(I)=2$ then $\sdepth(S/I)=1$.

(c) If $g(I)=3$ and $c(I)=3$, then $I'$ satisfies the hypothesis of Proposition $2.1$ and $\sdepth(S/I)=\sdepth(S/I')$.

(d) If $g(I)=:m > 3$ and $c(I)=3$, for every subset $\sigma\subset [m]$ with $|\sigma|=3$ we consider the ideal $I_{\sigma}=(u_{\sigma(1)},u_{\sigma(2)},u_{\sigma(3)})$, where $G(I)=\{u_1,\ldots,u_m\}$. If $\sdepth(S/I_{\sigma})=1$ 
for all $\sigma$ then $\sdepth(S/I)=1$. Otherwise, if $I\neq (I:(x_1,x_2,x_3))$ then $\sdepth(S/I)=0$.
\end{teor}

\begin{proof}
(a,b) If $g(I)=1$ then $I$ is principal, and therefore $\sdepth(I)=2$. If $g(I)=2$ or 
$c(I)=2$, by Proposition $1.4$, it follows that $\sdepth(S/I)=1$.

(c) $I'$ is a monomial ideal minimally generated by three monomials $u_1,u_2,u_3$ with $GCD(u_1,u_2,u_3)=1$. One can easily see that
$\sqrt{I'}$ must be in one of the three cases of Proposition $2.1$. By Proposition $1.3(1)$, $\sdepth(S/I)=\sdepth(S/I')$.

(d) Assume $I\neq I^{sat}$ and choose a monomial $w\in (I:(x_1,x_2,x_3))\setminus I$. Therefore, there exists some minimal generators $u_1,u_2,u_3\in G(I)$ such that $u_j| x_jw$ for any $j\in [3]$. Note that $u_1,u_2$ and $u_3$ are distinct. Indeed, if we assume by contradiction that $u_1=u_2$, since $u_1| x_1w$ and $u_1| x_2w$ it follows that $u_1|w $, which is absurd! Thus, $w\in (u_1,u_2,u_3)^{sat}\setminus (u_1,u_2,u_3)$ and therefore, by $1.5$, $\sdepth(S/(u_1,u_2,u_3))=0$.
\end{proof}

\begin{teor}
Let $I\subset S:=K[x_1,x_2,x_3]$ be a monomial ideal with $g(I)=3$ and $c(I)=3$. Then $\sdepth(I)=2$.
\end{teor}

\begin{proof}
By Proposition $1.3(2)$, we can assume that $I=I'$. If $I$ is generated by powers of variables then, by \cite[Theorem 1.3]{mir} or \cite[Proposition 3.8]{hvz}, it follows that $\sdepth(I)=2$. If this is not the case, we must consider several cases.

(1) If $\sqrt{I}=(x_1,x_2)$, then $I=(x_1^{a}, x_2^{b}, x_1^{a_1}x_2^{b_1}x_3^{c_1})$, where $a,b,a_1,b_1,c_1$ are some positive integers such that $a>a_1$ and $b>b_1$. By Lemma $1.2$, we have:
\[ I = ((x_1^{a_1})^c\cap I)\oplus x_1^{a_1}(I:x_1^{a_1}) = \bigoplus_{j=0}^{a_1-1}x_1^jx_2^b K[x_2,x_3] \oplus x_1^{a_1}(x_1^{a-a_1},x_2^b,x_2^{b_1}x_3^{c_1}).\]
We denote $J=(x_1^{a-a_1},x_2^b,x_2^{b_1}x_3^{c_1})$. By Lemma $1.2$, we have:
\[ J = ((x_2^{b_1})^c\cap J) \oplus x_2^{b_1}(J:x_2^{b_1}) = \bigoplus_{j=0}^{b_1-1}x_1^{a-a_1} x_2^j K[x_1,x_3] \oplus x_2^{b_1}(x_1^{a-a_1},x_2^{b-b_1},x_3^{c_1}).\]
By \cite[Theorem 1.3]{mir} or \cite[Proposition 3.8]{hvz}, $\sdepth((x_1^{a-a_1},x_2^{b-b_1},x_3^{c_1}))=2$, and therefore, from the above decompositions, it follows that $\sdepth(J)=2$ and thus $\sdepth(I)=1$.

(2) If $\sqrt{I}=(x_1,x_2x_3)$, then $I=(x_1^{a_1}, x_2^{b_1}x_3^{c_1}, x_1^{a_2}x_2^{b_2}x_3^{c_2})$, where $a_1,a_2,b_1,b_2,c_1,c_2$ are some positive integers such that $a_1>a_2$ and $b_1>b_2$ or $c_1>c_2$. It is enough to consider the case $b_1>b_2$. By Lemma $1.2$, 
\[ I = ((x_2^{b_2})^{c}\cap I)\oplus x_2^{b_2}(I:x_2^{b_2}) = \bigoplus_{j=0}^{b_2-1}x_1^{a_1}x_2^j K[x_1,x_3] \oplus x_2^{b_2}(x_1^{a_1}, x_2^{b_1-b_2}x_3^{c_1}, x_1^{a_2}x_3^{c_2}).\]
We denote $J=(x_1^{a_1}, x_2^{b_1-b_2}x_3^{c_1}, x_1^{a_2}x_3^{c_2})$. If $c_1<c_2$, by Lemma $1.2$, we have:
\[ J = ((x_3^{c_1})^c\cap J)\oplus x_3^{c_1}(J:x_3^{c_1}) = \bigoplus_{j=0}^{c_1-1}x_1^{a_1}x_3^j K[x_1,x_2] \oplus x_3^{c_1}(x_1^{a_1}, x_2^{b_1-b_2},x_1^{a_2}x_3^{c_2-c_1}).\]
By (1), $\sdepth((x_1^{a_1}, x_2^{b_1-b_2},x_1^{a_2}x_3^{c_2-c_1}))=2$ and thus, from the above decompositions, it follows that $\sdepth(I)=2$.

(3) If $\sqrt{I}=(x_1x_2,x_1x_3,x_2x_3)$, then $I=(x_1^{a_1}x_2^{b_1}, x_1^{a_2}x_3^{c_1}, x_2^{b_2}x_3^{c_2})$, where $a_1,a_2,b_1,b_2,c_1,c_2$ are some positive integers. We may assume $a_1\geq a_2$. By Lemma $1.2$, it follows that
\[ I=((x_1^{a_2})^c\cap I)\oplus x_1^{a_2}(I:x_1^{a_2}) = \bigoplus_{j=0}^{a_2-1}x_1^jx_2^{b_2}x_3^{c_2}K[x_2,x_3]\oplus x_1^{a_2}(x_1^{a_1-a_2}x_2^{b_1}, x_3^{c_1}, x_2^{b_2}x_3^{c_2}).\]
We denote $J=(x_1^{a_1-a_2}x_2^{b_1}, x_3^{c_1}, x_2^{b_2}x_3^{c_2})$. If $a_1=a_2$, then $c(J)=2$ and therefore $\sdepth(J)=2$.
Suppose $a_1>a_2$. If $c_1\leq c_2$ it follows that $g(J)=2$ and therefore $\sdepth(J)=2$. Otherwise, $J$ is an ideal in the case (2) and thus $\sdepth(J)=2$. From the above decomposition of $I$, it follows that $\sdepth(I)=2$.
\end{proof}

\begin{exm}
Let $I=(x_1^3,x_2^2x_3^2,x_1x_2^3x_3)$. We have:
\[ I = ((x_2^{2})^{c}\cap I)\oplus x_2^{2}(I:x_2^{2})=x_1^{3} K[x_1,x_3]\oplus x_1^{3}x_2 K[x_1,x_3]\oplus x_2^2(x_1^3,x_3^2,x_1x_2x_3),\]
On the other hand,
\[ (x_1^3,x_3^2,x_1x_2x_3) = x_3^2 K[x_2,x_3] \oplus x_1 (x_1^2,x_3^2,x_2x_3) = x_3^2 K[x_2,x_3] \oplus x_1^3K[x_1,x_2]\oplus x_1^2x_3 (x_1^2,x_2,x_3).\] We obtain the following Stanley decomposition of $I$:
\[ I = x_1^{3} K[x_1,x_3]\oplus x_1^{3}x_2 K[x_1,x_3]\oplus x_2^2x_3^2 K[x_2,x_3] \oplus x_2^2x_1^3K[x_1,x_2] \oplus \]
\[ \oplus x_1^2x_2^2x_3(x_2K[x_2,x_3]\oplus x_1x_2K[x_2,x_3] \oplus x_1^2K[x_1,x_3]\oplus x_3K[x_1,x_2]\oplus x_1^2x_2x_3K[x_1,x_2,x_3] ).\]
\end{exm}

\pagebreak

\begin{obs}
\emph{The conclusion of the Theorem $2.4$ is not true for $g(I)\geq 4$. Indeed, consider the ideal $I=(x_1^2,x_2^2,x_3^2,x_1x_2)\subset S$. We consider a Stanley decomposition of $I$, $\mathcal D: I= \bigoplus_{i=1}^r u_i K[Z_i]$. We assume that $\sdepth(\mathcal D)=2$, and thus $|Z_i|\geq 2$ for all $i$. 
Since $x_1^2,x_2^2,x_3^2,x_1x_2$ are the minimal generators of $I$, we can assume that $u_i=x_i^2$ 
for $i\in [3]$ and $u_4=x_1x_2$. If $\{x_2,x_3\}\subset Z_1$ it follows that $x_1\notin Z_2$, otherwise $x_1^2x_2^2\in x_1^2K[Z_1]\cap x_2^2K[Z_2]$. Similarly, $x_1\notin Z_3$. Therefore, $Z_2=Z_3=\{x_2,x_3\}$. But in this case, $x_2^2x_3^2\in x_2^2K[Z_2]\cap x_3^2K[Z_3]$, a contradiction.} 

\emph{Thus, we may assume $Z_1=\{x_1,x_2\}$, $Z_2=\{x_2,x_3\}$ and $Z_3=\{x_3,x_1\}$. If $x_1\in Z_4$ then $x_1^2x_1 \in x_1^2K[x_1,x_2]\cap x_1x_2K[Z_4]$, a contradiction. Thus $x_1\notin Z_4$. Similarly, $x_2\notin Z_4$ and therefore, $|Z_4|\leq 1$. This contradict $\sdepth(\mathcal D)=2$. It follows that $\sdepth(I)=1$.}

\emph{An interesting problem, put by Jurgen Herzog, is to compute the Stanley depth of the powers of the maximal ideal $(x_1,\ldots,x_n)\subset K[x_1,\ldots,x_n]$. We consider the case $n=3$ and $k\geq 2$. Let $I=(x_1,x_2,x_3)^k\subset K[x_1,x_2,x_3]$. Using the same argument as above, we can easily see that $\sdepth(I)=1$.}
\end{obs}

\begin{lema}
Let $I\subset S:=K[x_1,x_2,x_3]$ be a monomial ideal such that $I=I'$. For $j\in [3]$ we denote $I_j:=I\cap S_j:=K[Z_j]$, where $Z_j=\{x_1,x_2,x_3\}\setminus \{x_j\}$. If $I^{sat} = I$ then there exists some $j\in [3]$ such that $I_j^{sat}=I_j$.
\end{lema}

\begin{proof}
If $I=S$ there is nothing to prove, so we can assume $I\neq S$. Since $I^{sat}=I$, it follows that $\mathbf{m}=(x_1,x_2,x_3)\notin Ass(S/I)$. Since $I=I'$, it follows that $(x_j)\notin Ass(S/I)$ for all $j\in [n]$. We denote $\mathbf{m}_j$ the irrelevant ideal of $S_j$. Therefore, $Ass(S/I)\subset \{\mathbf{m}_1,\mathbf{m}_2,\mathbf{m}_3\}$. 

Thus, we can find a decomposition $I=\bigcap_{j=1}^3 Q_j$ such that $Q_j$ is $\mathbf{m}_j$-primary or $Q_j=S$ for all $j\in [3]$. It follows that $I_k=\bigcap_{j=1}^3 (Q_j\cap S_k)$. We assume $Q_1=(x_2^a,x_3^b,\ldots)$, $Q_2=(x_1^c,x_3^d,\ldots)$ and $Q_3=(x_1^e,x_2^f,\ldots)$ where $a,b,c,d,e,f$ are some nonnegative integers. By reordering the variables, we can assume $a\geq f$. It follows that $Q_1\cap S_3 = (x_2^a) \subset Q_3$. Therefore, $I_3 = (x_2^a)\cap (x_3^d) = (x_2^ax_3^d)$ is principal and moreover, $I_3=I_3^{sat}$.
\end{proof}

\begin{prop}
Let $I\subset S:=K[x_1,x_2,x_3]$ be a monomial ideal which is not principal. If $I=I^{sat}$ then $\sdepth(I)=2$.
\end{prop}

\begin{proof}
Since $sdepth(I)=sdepth(I')$ and $I^{sat} = vI'^{sat}$, where $v=GCD(u|\;u\in G(I))$, we may assume $I=I'$. If $c(I)=2$ or $g(I)=2$ there is nothing to prove, and therefore, we may assume that $c(I)=3$ and $g(I)\geq 3$. We use the notations from Lemma $2.7$. By $2.7$, we can assume that $I_1^{sat}=I_1$ and $I_1$ is principal. Thus $\sdepth(I_1)=2$. We write
$I=I_1\oplus x_1(I:x_1)$. Obviously, $I\subsetneq (I:x_1)$ and $(I:x_1)^{sat}=(I:x_1)$. We can use the same procedure for $(I:x_1)$. We obtain a chain of ideals which must stop. Thus, we obtain a Stanley decomposition of $I$ with its Stanley depth equal to $2$.
\end{proof}

\begin{obs}
\emph{The converse of $2.8$ is not true, take for instance $\mathbf{m}=(x_1,x_2,x_3)\subset K[x_1,x_2,x_3]$. By Stanley's conjecture, it would be expected that any monomial ideal $I\subset S=K[x_1,\ldots,x_n]$ with $\sdepth(I)=1$ has also $\depth(I)=1$. But $\depth(I)=1$ if and only if $\depth(S/I)=0$, which is equivalent with $I^{sat}\neq I$. So, by $2.8$, it follows that if $I\subset S:=K[x_1,x_2,x_3]$ has $\sdepth(I)=1$ then $\depth(I)=1$.
Unfortunatelly, a similar result to Lemma $2.7$ is not true for $n\geq 4$.}

\emph{Let $S=K[x_1,x_2,x_3,x_4]$, $Q_1=(x_2^3,x_3^2,x_4)$, $Q_2=(x_1^3,x_3,x_4^2)$, $Q_3=(x_1^2,x_2,x_4^3)$, $Q_4=(x_1,x_2^2,x_3^3)$ and $I=Q_1\cap Q_2\cap Q_3\cap Q_4$. With similar notations as those of $2.7$, one can easily see that $I_k=I\cap S_k = \bigcap_{j=1}^4 (Q_j\cap S_k)$ is a reduced primary decomposition of $I_k$. In particular, $\mathbf{m}_k = \sqrt{Q_j\cap S_j} \in Ass(S_j/I_j)$ and thus $I_j^{sat}\neq I_j$.}
\end{obs}

\vspace{2mm} \noindent {\footnotesize
\begin{minipage}[b]{15cm}
 Mircea Cimpoeas, Institute of Mathematics of the Romanian Academy, Bucharest, Romania\\
 E-mail: mircea.cimpoeas@imar.ro
\end{minipage}}
 
\end{document}